\documentclass[12pt]{amsart}
\usepackage{amssymb}
\usepackage{graphicx}
\usepackage{amsmath}

\newtheorem{theorem}{Theorem}[section]

\newtheorem{lemma}{Lemma}[section]

\newtheorem{proposition}{Proposition}[section]

\newtheorem{remarks}{Remarks}[section]

\numberwithin{equation}{section}

\begin{document}
\title[Decay of the local energy...]{Decay of the local energy for the solutions of the critical Klein-Gordon equation}
\author{Ahmed Bchatnia}
\address{UR  Analyse Non-Lin\'eaire et G\'eom\'etrie, UR13ES32, Department of Mathematics, Faculty of Sciences of Tunis, University of Tunis El Manar, Tunisia}
\email{ahmed.bchatnia@fst.utm.tn}
\author{Naima Mehenaoui}
\address{UR  Analyse Non-Lin\'eaire et G\'eom\'etrie, UR13ES32, Department of Mathematics, Faculty of Sciences of Tunis, University of Tunis El Manar, Tunisia}
\email{naima.mehenaoui@fst.utm.tn}
\date{\today}
\subjclass[2000]{35B35, 35L05.}
\keywords{Klein-Gordon, wave, exponential decay.}

\begin{abstract}
In this paper, we prove the exponential decay of local
energy for the Klein-Gordon equation with localized critical nonlinearity. The proof relies on generalized Strichartz estimates, and semi-group of Lax-Phillips.

{R\'{e}sum\'{e}}: Dans cet article, on d\'{e}montre la d\'{e}croissance
exponentielle de l'\'{e}nergie locale des solutions de l'\'{e}quation  de Klein-Gordon, avec une nonlin\'{e}arit\'{e} critique localis\'{e}e. La preuve est bas\'{e}e sur les in\'{e}galit\'{e}s de
Strichartz g\'{e}n\'{e}ralis\'{e}es, et le semi-groupe de Lax-Phillips.

\end{abstract}

\maketitle

\tableofcontents

\section{Introduction and position of the problem}
In this paper, we are interesting to the following system:
\begin{equation}\label{1.1}
\left\{
\begin{array}{l}
\square u+\chi_{1}u+\chi_{2} u^{5}=0\ \text{on\ }\ \mathbb{R}\times \mathbb{R}^{3}
\\
u(0,x)=u^{0}(x)\in H^{1}(\mathbb{R}^{3} )\ \text{and}\ \partial
_{t}u(0,x)=u^{1}(x)\in L^{2}(\mathbb{R}^{3} ),%
\end{array}%
\right.
\end{equation}%
where $\square=\partial_{t}^{2}-\triangle$, $\chi_{1}$ and $\chi_{2}$ are positives functions, of class $C^{1}$, with compact support such that $supp \chi_{1}\cup supp \chi_{2}\subset B_{R}$ for some $R>0$ and satisfying
\begin{equation}\label{cond}
x\cdot\nabla\chi_{1}(x)\leq 0 \quad \text{and}\quad x\cdot\nabla\chi_{2}(x)\leq 4,\; \forall x\in \mathbb{R}^3.
\end{equation}
 We denote $%
H=H^{1}(\mathbb{R}^{3} )\times L^{2}(\mathbb{R}^{3} )$ with respect to the norm
\begin{equation*}
\left\Vert \left( \varphi _{1},\varphi _{2}\right) \right\Vert
_{H}^{2}=\int_{\mathbb{R}^{3}}\left(\left\vert \nabla \varphi _{1}\right\vert
^{2}+\left\vert \varphi _{2}\right\vert ^{2}\right)dx.
\end{equation*}%
Global existence and uniqueness of solutions to the Cauchy problem (\ref{1.1}) has been extensively studied in the last decades. Thanks to the works of \cite{12,13} and \cite{27,28}, it is by now well known that for every initial data  $(u^{0},u^{1})\in H^{1}(\mathbb{R}^{3})\times  L^{2}(\mathbb{R}^{3})$; system (\ref{1.1})  admits a unique solution $u$ in the \textquotedblleft Shatah-Struwe\textquotedblright\
class, that is%
\begin{equation*}
u\in C(\ \mathbb{R},H^{1}(\mathbb{R}^{3} ))\cap L_{loc}^{5}(\ \mathbb{R}%
,L^{10}(\mathbb{R}^{3} )),\text{ }\partial _{t}u\in C(\ \mathbb{R},L^{2}(\mathbb{R}^{3} )).
\end{equation*}%
The global energy of $u$ at time $t$ is defined by%
\begin{equation}
E(u(t))=\frac{1}{2}\int_{\mathbb{R}^{3} }\bigskip \left( \left\vert \partial
_{t}u\left( t\right) \right\vert ^{2}+\left\vert \nabla _{x}u\left( t\right)
\right\vert ^{2}+\chi_{1}(x)|u|^{2}\right) dx+\frac{1}{6}\int_{\mathbb{R}^{3} }\chi_{2}(x)\left\vert
u\left( t\right) \right\vert ^{6}dx,
\end{equation}%
which is time independent.
\medskip\\
We define the local energy by%
\begin{equation}
E_{R }(u(t))=\frac{1}{2}\int_{B_{R }}\bigskip \left(
\left\vert \partial _{t}u\left( t\right) \right\vert ^{2}+\left\vert \nabla
_{x}u\left( t\right) \right\vert ^{2}+\chi_{1}(x)|u|^{2}\right) dx+\frac{1}{6}\int_{B_{R}}\chi_{2}(x)\left\vert u\left( t\right) \right\vert ^{6}dx;
\end{equation}%
where $B_{R }$ is a ball of radius $R$.
\medskip\\
For every $t\in \mathbb{R},$ we define the nonlinear Klein Gordon operator $U(t)$ by%
\begin{equation*}
\begin{array}{cccc}
U(t): & H & \longrightarrow & H \\
& \left( \varphi _{1},\varphi _{2}\right) & \longmapsto & U\left( t\right)
\left( \varphi _{1},\varphi _{2}\right) =\left( u\left( t\right) ,\partial
_{t}u\left( t\right) \right)%
\end{array}%
\newline
\end{equation*}%
where $u$ is the solution of (\ref{1.1}) in the \textquotedblleft
Shatah-Struwe\textquotedblright\ class\ with initial data $\varphi =\left(
\varphi _{1},\varphi _{2}\right) .$\\*[0pt]
$\left( U\left( t\right) \right) _{t\in \mathbb{R}}$ forms a one parameter
continuous group on $H,$ to which we will refer as the nonlinear group.
\medskip\\
The main result of this work is to prove that the decay of the local energy of the solutions of (\ref{1.1})  is of exponential type. More precisely we have the following theorem:
\begin{theorem}[Exponential decay of the local energy]\label{exp}
For all $R>0$, there exist $\alpha>0$ and $c>0$ such that
\begin{equation}
E_{R}(u(t))\leq C e^{-\alpha t}E(0)
\end{equation}
holds for every $u$ solution to (\ref{1.1})  with initial data $(u^{0}, u^{1})\in H$ supported in $B_{R}$.
\end{theorem}
We note that a large number of works have been devoted to the local energy decay for wave equation, we
principally quote the works of W. Strauss \cite{26}, and Jeng-Eng. Lin \cite{16}. Moreover, we mention the work of the first author \cite{39}, which treat the case of critical wave equations outside a convex obstacle.\medskip \\
For Klein-Gordon equation, the literature is less provide. We quote essentially the result of C. Morawetz \cite{19} and the result of stabilization of B. Dehman and P. G\'{e}rard \cite{6}. Furthermore, a recent result by M. Malloug
\cite{38} which establishes an exponential decay of the local energy for the damped Klein-Gordon equation in exterior domain and R-S-O. Nunes and W-D. Bastos \cite{36} obtains polynomial decay of the local energy for the linear Klein Gordon equation.

The rest of this paper is organized as follows.\\
In section 2,  we prove by an argument inspired from the works of \cite{32, 4} that the Strichartz norms of the solutions to (\ref{1.1})  are global in time.\\
In section 3, we prove the exponential decay of the localized linear Klein-Gordon equation. For this purpose, we introduce the Lax-Phillips semigroup $Z_{KG}(t)$ and an argument inspired from the work of \cite{36}.\\
Section 4 is devoted to prove the main result of this paper. By combining the results obtained in section 2 and 3, dealing with the nonlinear term $\chi_{2}u^{5}$ as a source term and using the Gronwall lemma in crucial way, we obtain then the exponential decay of the local energy for the solutions of (\ref{1.1}).
\section{Strichartz norms global in time}

The goal of this section is to prove that Strichartz norms for the solutions of (\ref{1.1}) are global in time, we recall the following theorem due to J. Zhang and J. Zheng.

\begin{theorem} \cite[Zhang--Zheng]{37} \label{thm2} \\
Let $(X,g)$ be a nontrapping scattering manifold of dimension $n\geq 3$. Suppose that $u$ is the solution to the Cauchy problem:%
\begin{equation*}
(S)\left\{
\begin{array}{l}
\partial_{t}^{2}u-\triangle_{g}u+u=F\left( t,z\right) \text{ }(t,z)\in I\times X \\
u(0)=u_{0}(z)\ \text{and }\ \partial_{t}u(0)=u_{1}(z)%
\end{array}%
\right.
\end{equation*}%
For some initial data $u_{0}\in H^{s}$, $u_{1}\in H^{s-1}$, and the time interval $I\subseteq \mathbb{R}$, then
\begin{eqnarray*}
  \|u(t,z)\|_{L_{t}^{q}(I;L_{z}^{r}(x))}+\|u(t,z)\|_{C(I;H^{s}(x))}\leq \|u_{0}\|_{H^{s}(X)}+ \|u_{1}\|_{H^{s-1}(X)}+\|F\|_{L_{t}^{\widetilde{q}'}(I;L_{z}^{\widetilde{r}'}(x))},
\end{eqnarray*}
where the pairs $(q,r)$,$(\widetilde{q},\widetilde{r})\in [2,+\infty]^{2}$ satisfy the KG-admissible condition with $0\leq \theta \leq 1$
\begin{equation*}
\frac{2}{q}+\frac{n-1+\theta}{r}\leq\frac{n-1+\theta}{2},\quad (q,r,n,\theta)\neq (2,\infty,3,0)
\end{equation*}
and the gap condition
\begin{equation*}
 \frac{1}{q}+\frac{n+\theta}{r} =\frac{n+\theta}{2}-s=\frac{1}{\widetilde{q}'}+\frac{n+\theta}{\widetilde{r}'}-2
\end{equation*}
\end{theorem}
From Theorem \ref{thm2} we deduce the following proposition.
\begin{proposition}\label{prop1}
 Given $(q,r)\in [2,+\infty]$ with $\tilde{q}=1$, $\tilde{r}=2$, $s=1$ and $n=3$ satisfy the KG-admissible condition with $0\leq \theta \leq 1$
\begin{equation*}
\frac{2}{q}+\frac{2+\theta}{r}\leq\frac{2+\theta}{2},\quad (q,r,n,\theta)\neq (2,\infty,3,0)
\end{equation*}
and the gap condition
\begin{equation*}
 \frac{1}{q}+\frac{3+\theta}{r} =\frac{1+\theta}{2}.
\end{equation*}
Then for every $T\geq 0$, for every $u$ solution of $(S)$ we have
\begin{equation}\label{}
  \|u\|_{L^{q}([0,T],L^{r}(\mathbb{R}^{3}))}\leq E(u(0))^{{1}/{2}}+\left\Vert F\left( t,x\right) \right\Vert _{L^{1}(%
\left[ 0,T\right] ,L^{2}(\mathbb{R}^{3}))}
\end{equation}
\end{proposition}
\begin{proof}
For $T>0$, we define a cutoff function by\\
$$\chi_{T}=\left\{
  \begin{array}{ll}
    1 & \quad \text{if}\; 0\leq t\leq T \\
    0 & \quad \text{if not};
  \end{array}
\right.$$\\
and we consider, $v_{T}$ the solution of the system\begin{equation*}
(S)\left\{
\begin{array}{l}
\square v_{T}+\chi_{T}F\left( t,x\right)=0 \text{\ on \ }\mathbb{R}\times \mathbb{R}^{3} \\
(v_{T}(0,x),\partial
_{t}v_{T}(0,x))=(u^{0}(x),u^{1}(x))%
\end{array}%
\right.
\end{equation*}%
where $u$ is the solutions of $(S)$ with initial data $(u^{0}(x),u^{1}(x))$ in $H$. By virtue of local time Strichartz of \cite{37}, we have $\chi_{T}F(t,x)\in L^{1}(\mathbb{R},L^{2}(\mathbb{R}^{3}))$ and thanks to  Theorem \ref{thm2}, we deduce
\begin{equation*}
\|v_{T}\|_{L^{q}(\mathbb{R},L^{r}(\mathbb{R}^{3}))}\leq E(u(0))^{{1}/{2}}+\left\Vert \chi_{T}F\left( t,x\right) \right\Vert _{L^{1}(\mathbb{R} ,L^{2}(\mathbb{R}^{3}))},
\end{equation*}
and therefore
\begin{equation*}
\|v_{T}\|_{L^{q}([0,T],L^{r}(\mathbb{R}^{3}))}\leq E(u(0))^{{1}/{2}}+\left\Vert \chi_{T}F\left( t,x\right) \right\Vert _{L^{1}([0,T] ,L^{2}(\mathbb{R}^{3}))}.
\end{equation*}
Since $u=v_{T}$ on $[0,T]\times \mathbb{R}^{3}$, we obtain
\begin{equation*}
\|u\|_{L^{q}([0,T],L^{r}(\mathbb{R}^{3}))}\leq E(u(0))^{{1}/{2}}+\left\Vert \chi_{T}F\left( t,x\right) \right\Vert _{L^{1}([0,T] ,L^{2}(\mathbb{R}^{3}))}.
\end{equation*}
\end{proof}
Let us recall now the following bootstrap lemma (see \cite{4}).
\begin{lemma}\label{lemma1}
Let $M(t)$ be a nonnegative continuous function in $\left[ 0,T\right] $ such
that, for every$\ t\in \left[ 0,T\right] ,\ $%
\begin{equation*}
M(t)\leq a+bM(t)^{\theta },
\end{equation*}%
where $a,b>0$ and $\theta >1$ are constants such that,
\begin{equation*}
a<\left(1-\frac{1}{\theta }\right) \frac{1}{(\theta b)^{1/\theta -1}},\text{ \ \ \ }%
M(0)\leq \frac{1}{(\theta b)^{1/\theta -1}}.
\end{equation*}%
Then for every $t$ $\in \left[ 0,T\right] $, we have
\begin{equation*}
M(t)\leq \frac{\theta }{\theta -1}a.
\end{equation*}
\end{lemma}
Let%
\begin{equation*}
e(t)=\int\limits_{\substack{ \left\vert x\right\vert \leq t \\ x\in \mathbb{R}^{3}
}}\left[ \frac{1}{2}\left\vert \nabla _{x}u(t,x)\right\vert ^{2}+\frac{1}{2}%
\left\vert \partial _{t}u(t,x)\right\vert ^{2}+\frac{1}{2}\left\vert
u(t,x)\right\vert ^{2}+\frac{1}{6}\chi (x)\left\vert
u(t,x)\right\vert ^{6}\right] dx.
\end{equation*}%
\begin{lemma}\label{lemma2}
There exists $D>0$ such that, for all $b>a>R,$ for every solution $u$ to $%
\square u+\chi_{1}u+\chi_{2} (x)u^{5}=0,\ $with $u\in C(\left[ a,b\right] ,H^{1}(\mathbb{R}^{3}
))\cap L^{5}(\left[ a,b\right] ,L^{10}(\mathbb{R}^{3} )),$ $\partial _{t}u\in C(%
\left[ a,b\right] ,L^{2}(\mathbb{R}^{3} ))$ we have%
\begin{equation*}
\int_{_{\substack{ \left\vert x\right\vert \leq a  \\ x\in \mathbb{R}^{3} }}}\chi_{2}
(x)\left\vert u(a,x)\right\vert ^{6}dx\leq D\left[ \frac{b}{a}%
(e(b)+e(b)^{1/3})\right] .
\end{equation*}
\end{lemma}
 The general scheme of the proof of Lemma \ref{lemma2} is the same in \cite{32}, but by choosing an adequate multiplier for the equation of Klein-Gordon.\newline
\begin{proof}
We use the notations of \cite{32}. Let
\begin{equation*}
K_{a}^{b}=\left\{ (t,x)\in \mathbb{R}\times \mathbb{R}^{3},a\leq t\leq b,\left\vert x\right\vert \leq t\right\}
\end{equation*}%
the truncated light cone$,$
\begin{equation*}
M_{a}^{b}=\left\{ (t,x)\in \mathbb{R}\times \mathbb{R}^{3},a\leq t\leq b,\left\vert x\right\vert =t\right\}
\end{equation*}%
the `mantle' associated with $K_{a}^{b},$ and
\begin{equation*}
D(t)=\left\{ (t,x)\in \mathbb{R}\times \mathbb{R}^{3},\left\vert x\right\vert \leq t\right\}
\end{equation*}%
its spacelike sections. We note that%
\begin{equation*}
\partial K_{a}^{b}=D(a)\cup D(b)\cup M_{a}^{b}.\newline
\end{equation*}%
We start with initial data in $\left( C_{0}^{\infty }\left( \mathbb{R}^{3}\right)
\right)^{2}$, hence the associated solution is of class $C^{\infty }$.\newline
Multiplying equation (\ref{1.1}) by $\partial _{t}u,$ we obtain%
\begin{equation}\label{2.2}
div_{t,x}\left( \frac{1}{2}\left\vert \nabla _{x}u(t,x)\right\vert ^{2}+\frac{1}{2}%
\left\vert \partial _{t}u(t,x)\right\vert ^{2}+\frac{1}{2}\chi_{1}\left\vert
u(t,x)\right\vert ^{2}+\frac{1}{6}\chi_{2}(x)\left\vert
u(t,x)\right\vert ^{6},-\partial _{t}u\nabla _{x}u\right) =0,
\end{equation}%
then we integrate (\ref{2.2}) over the truncated cone $K_{a}^{b}$ to obtain the
classical energy identity%
\begin{equation}\label{2.3}
e(b)-e(a)=\int_{M_{a}^{b}}\left( \frac{1}{2}\left\vert \frac{x}{t+1}\partial
_{t}u+\nabla _{x}u\right\vert ^{2}+\frac{1}{2}\chi_{1}\left\vert
u(t,x)\right\vert ^{2}+\frac{1}{6}\chi_{2}(x)\left\vert
u\right\vert ^{6}\right) \frac{d\sigma }{\sqrt{2}}.
\end{equation}%
Moreover, multiplying (\ref{1.1}) by $Lu=(-t \partial _{t}+x\cdot \nabla +1)u$ we
obtain%
\begin{equation}\label{2.4}
div_{t,x}\left( tQ+\partial _{t}uu,-tP\right) +\left( \frac{2}{3}\chi_{2}-\frac{1}{6}x\cdot \nabla \chi_{2}
\right) u^{6}+|\partial _{t}u|^{2}+|\nabla_{x}u|^{2}-\left(x.\nabla \chi_{1}\right) u^{2}=0,
\end{equation}%
where%
\begin{eqnarray*}
Q &=&-\left[ \frac{1}{2}(\left\vert \nabla _{x}u\right\vert ^{2}+\left\vert \partial
_{t}u\right\vert ^{2}+\chi_{1}|u|^{2})+\frac{1}{6}\chi_{2}(x)u^{6}\right] +\partial _{t}u\frac{x}{t}%
\cdot \nabla _{x}u \\
\text{and }\\P &=&\frac{x}{t}\left( \frac{1}{2}\left( \left\vert \partial _{t}u\right\vert
^{2}-\left\vert \nabla _{x}u\right\vert ^{2}-\chi_{1}|u|^{2}\right) -\frac{1}{6}\chi_{2}(x)u^{6}\right) +\nabla _{x}u\left(-\partial _{t}u+\frac{x}{t}\cdot \nabla _{x}u+\frac{u}{%
t}\right) .
\end{eqnarray*}%
Integrating (\ref{2.4}) over $K_{a}^{b}$ we obtain%
\begin{eqnarray}\label{2.5}
0 &=&\int_{D(b)}\left( b Q+(\partial _{t}u)u\right) dx-\int_{D(a)}\left( a Q+(\partial _{t}u)u\right) dx
\\
&&-\int_{M_{a}^{b}}\left( t Q+\partial _{t}uu+x\cdot P\right) \frac{d\sigma }{\sqrt{2}}%
+\int_{ K_{a}^{b}}\left( |\partial_{t}u|^{2}+|\nabla_{x}u|^{2}\right) dxdt  \notag
\\
&&+\int_{ K_{a}^{b}}\left[ \left( \frac{2}{3}\chi_{2}-\frac{1}{6}x\cdot \nabla \chi_{2}\right) u^{6}-(x.\nabla \chi_{1})u^{2}\right]  dxdt
\notag \\
&=&I+II+III+IV+V  \notag
\end{eqnarray}%

We start with the term $III$. Since $t=\left\vert x\right\vert $ on $%
M_{a}^{b},$ we can write%
\begin{equation*}
III=\int_{M_{a}^{b}}\left\vert x\right\vert\left( \chi_{1}|u|^{2}+\frac{1}{3}\chi_{2}u^{6}\right) \frac{d\sigma}{\sqrt{2}}-\int_{M_{a}^{b}}u\frac{x\cdot
\nabla _{x}u}{\left\vert x\right\vert }+(\partial _{t}u)u\frac{d\sigma }{%
\sqrt{2}}.
\end{equation*}%
We parameterize $M_{a}^{b}$ by%
\begin{equation*}
y\longmapsto (\left\vert y\right\vert ,y),\text{ \ \ }a\leq
\left\vert y\right\vert \leq b,
\end{equation*}%
and let $v(y)=u(\left\vert y\right\vert ,y)$. Then%
\begin{equation*}
d\sigma =\sqrt{2}dy\text{ \ \ and \ \ }y\cdot \frac{\nabla v}{\left\vert
y\right\vert }=\frac{x\cdot \nabla _{x}u}{\left\vert x\right\vert }+\partial
_{t}u.
\end{equation*}%

By integrating by parts, one sees that%
\begin{equation*}
\int_{\substack{ y\in \mathbb{R}^{3}  \\ a\leq \left\vert y\right\vert \leq b}}v%
\frac{y\cdot \nabla v}{\left\vert y\right\vert }dy=\frac{1}{2}\int
_{\substack{ y\in \mathbb{R}^{3}  \\ \left\vert y\right\vert =b}}v^{2}d\sigma -\frac{%
1}{2}\int_{\substack{ y\in \mathbb{R}^{3}  \\ \left\vert y\right\vert =a}}%
v^{2}d\sigma -\int_{\substack{ y\in \mathbb{R}^{3}  \\ a\leq \left\vert y\right\vert
\leq b}}\frac{v^{2}}{\left\vert y\right\vert }dy.
\end{equation*}%
So if we switch back to the original coordinates, we have%
\begin{equation}\label{2.6}
III=\int_{M_{a}^{b}}|x|\left( \chi_{1}u^{2}+\frac{u^{2}}{|x|^{2}}+\frac{1}{3}\chi_{2}u^{6}\right) \frac{d\sigma
}{\sqrt{2}}-\frac{1}{2}\int_{\partial D_{b}}u^{2}d\sigma +\frac{1}{2}%
\int_{\partial D_{a}}u^{2}d\sigma .
\end{equation}%
Now, we rewrite the first and second term of (\ref{2.5}) as%
\begin{equation}\label{2.7}
I+II=-H(b)+H(a)+\frac{1}{b}\int_{D_{b}}\left( x\cdot \nabla _{x}uu+\frac{3}{2}%
u^{2}\right) dx-\frac{1}{a}\int_{D_{a}}\left( x\cdot \nabla _{x}uu+\frac{3}{2}u^{2}\right) dx
\end{equation}%
where%
\begin{equation*}
H(t)=\int_{D(t)}t\left[ \frac{1}{2}\left\vert \frac{1}{t}Lu\right\vert ^{2}+%
\frac{1}{2}\left( \left\vert \nabla _{x}u\right\vert ^{2}-\left\vert \frac{%
x\cdot \nabla _{x}u}{t}\right\vert ^{2}+\chi_{1}(x)|u|^{2}\right) +\chi_{2}(x)\frac{\left\vert
u\right\vert ^{6}}{6}\right]+\frac{u^{2}}{t} dx.
\end{equation*}%
And, as above, integration by parts gives%
\begin{equation}\label{2.8}
\int_{D(t)}\left( x\cdot \nabla _{x}uu+\frac{3}{2}u^{2}\right) dx=\frac{t}{2}%
\int_{\partial D(t)}u^{2}d\sigma .
\end{equation}%
Therefore, we obtain from (\ref{2.5}), (\ref{2.6}), (\ref{2.7}) and (\ref{2.8})

\begin{eqnarray*}
H(a)&-&H(b)+\int_{M_{a}^{b}}t\left( \chi_{1}(x)u^{2}+\frac{u^{2}}{t^{2}}+\frac{1}{3}\chi_{2}(x)u^{6}\right) \frac{d\sigma}{\sqrt{2}}+\int_{ K_{a}^{b}}\left( |\partial_{t}u|^{2}+|\nabla_{x}u|^{2}\right) dxdt\\&  +&\int_{ K_{a}^{b}}\left[ \left( \frac{2}{3}-\frac{1}{6}x\cdot \nabla \chi_{2}(x)\right) u^{6}-\left( x.\nabla \chi_{1}(x)\right) u^{2}\right] dxdt \notag=0 \\
\end{eqnarray*}
Finally, using hypothesis \ref{cond} on $\chi_{1}$ and $\chi_{2}$ we have
\begin{equation}\label{2.9}
H(a)\leq H(b).
\end{equation}
And H\"{o}lder's inequality gives%
\begin{equation}\label{2.10}
\int_{D(t)}\chi_{2} (x)\frac{u^{6}}{6}dx\leq \frac{1}{t}H(t)\leq
C_{1}\left( e(t)+e(t)^{1/3}\right)
\end{equation}%

Consequently, dividing by $a$, we find from (\ref{2.9}) and (\ref{2.10})
\begin{equation}
\int_{D(a)}\chi_{2}(x)|u(a,x)|^{6}dx\leq D\left[ \frac{b}{a}%
\left( e(b)+e(b)^{1/3}\right) \right] .
\end{equation}%
\ \ \ Now, by density argument and the continuity of the nonlinear map
\begin{equation*}
\begin{array}{ccc}
F:\text{ \ \ \ \ }H & \longrightarrow & C\left( \left[ 0,T\right] ,H\right)
\\
\text{ \ \ \ \ \ \ \ \ \ \ }\left( \varphi ,\psi \right) & \longmapsto &
\left( u,\partial _{t}u\right) ,%
\end{array}%
\newline
\end{equation*}%
where $u$ is the solution to (\ref{1.1}) such that $\left( u,\partial _{t}u\right)
_{/t=0}=\left( \varphi ,\psi \right) $(see \cite{11}), the conclusion of
this lemma holds for every data in $H$.
\end{proof}

We come here to the final step, we apply Theorem \ref{thm2} and Lemma \ref{lemma2} to prove the following proposition.
\begin{proposition}
Let $u$ be a solution to (\ref{1.1}), then%
\begin{equation}\label{eq1}
\int\limits_{\mathbb{R}^{3}}\chi_{2} (x)\left\vert u(t,x)\right\vert ^{6}dx\underset{%
t\rightarrow \pm \infty }{\longrightarrow }0,
\end{equation}%
and for all $q>2$ and $r$ such that $\displaystyle\frac{1}{q}+\frac{1}{r}=%
\frac{1}{2}$ we have%
\begin{equation}\label{eq5}
u\in L^{q}\mathbb{(R}_{+},L^{3r}(\mathbb{R}^{3})).
\end{equation}
\end{proposition}

\begin{proof}
The classical energy identity (\ref{2.3}) shows that $e(t)$ is a nondecreasing bounded function of $t$, hence has a limit $L$ as $t \longrightarrow +\infty .$ Applying Lemma 2.1 with $a=T,$ $b=\varepsilon T$ and passing to
the limit as $T\longrightarrow $ $+\infty ,$ we obtain%
\begin{equation}\label{eq6}
\underset{T\longrightarrow +\infty }{\lim \sup }\int_{_{\substack{ %
\left\vert x\right\vert \leq T  \\ x\in \mathbb{R}^{3} }}}\chi_{2} (x)\left\vert
u(T,x)\right\vert ^{6}dx\leq D\varepsilon (L+L^{1/3})
\end{equation}%
for every $\varepsilon >0,$ hence the left-hand side of (\ref{eq6}) is $0$.
For $T$ large enough, we have $supp\chi_{2}\subset B(0,T)$ then
\begin{equation}\label{eq7}
   \int\limits_{\mathbb{R}^{3}}\chi_{2} (x)\left\vert u(t,x)\right\vert ^{6}dx=\int_{_{\substack{ %
\left\vert x\right\vert \leq T  \\ x\in \mathbb{R}^{3} }}}\chi_{2} (x)\left\vert
u(T,x)\right\vert ^{6}dx.
\end{equation}
Finally (\ref{eq1}) follows. Now applying Proposition \ref{prop1} with $q=4, r=12\medskip \medskip $ we have%
\begin{eqnarray*}
\left\Vert u\right\Vert _{L^{4}(\left[ T,S\right] ,L^{12}(\mathbb{R}^{3} ))}\medskip
&\leq &C\left( E(u(T))^{{1}/{2}}+\left\Vert \chi_{2} (x)u^{5}\right\Vert _{L^{1}(%
\left[ T,S\right] ,L^{2}(\mathbb{R}^{3} ))}\right)  \\
&\leq &C\left( E(u(0))^{{1}/{2}}+\left\Vert \chi_{2} (x)u\right\Vert _{L^{5}(\left[
T,S\right] ,L^{10}(\mathbb{R}^{3} ))}^{5}\right)
\end{eqnarray*}%
which yields by H\"{o}lder's inequality%
\begin{equation*}
\left\Vert u\right\Vert _{L^{4}(\left[ T,S\right] ,L^{12}(\mathbb{R}^{3} ))}\medskip
\leq C\left( E(u(T))^{{1}/{2}}+\left\Vert \chi_{2} u\right\Vert _{L^{\infty }(\left[
T,S\right] ,L^{6}(\mathbb{R}^{3} ))}\left\Vert u\right\Vert _{L^{4}(\left[ T,S\right]
,L^{12}(\mathbb{R}^{3} ))}^{4}\right) .
\end{equation*}
Then, by choosing $T$ large enough and using (\ref{eq1}), we conclude by applying lemma \ref{lemma1} that $u\in L^{4}(%
\mathbb{R}_{+},L^{12}(\mathbb{R}^{3} ))$ and by H\"{o}lder's inequality $u\in L^{5}(%
\mathbb{R}_{+},L^{10}(\mathbb{R}^{3} )).$\\
Finally (\ref{eq5}) follows by virtue of theorem \ref{thm2} .
\end{proof}

\section{Exponential decay of the local energy of localized linear Klein-Gordon equation}
The goal of this section is to prove the exponential decay of the local energy for the localized linear Klein-Gordon equation,%
\begin{equation}\label{sys2}
\left\{
\begin{array}{l}
\square u+\chi_{1}(x)u=0\quad \text{on}\ \mathbb{R}\times \mathbb{R}^{3}
\\
u(0,x)=u^{0}(x)\in H^{1}(\mathbb{R}^{3} )\ \text{and}\ \partial
_{t}u(0,x)=u^{1}(x)\in L^{2}(\mathbb{R}^{3} ),%
\end{array}%
\right.
\end{equation}%
where $\chi_{1} $ is a function of class $C^{1}$ with compact support such that $supp\chi_{1}\subset B_{R}$, for some $R>0$.\medskip\newline
For that, we prove the following theorem:
\begin{theorem}\label{thm3}
 Let $R>0$, there exist $\alpha>0$ and $c>0$ such that
\begin{equation}
E_{R}(u(t))\leq C e^{-\alpha t}E(0)
\end{equation}
holds for every $u$ solution to (\ref{sys2}) with initial data $(\varphi_{0},\varphi_{1})\in H$ supported in $B_{R}$.
\end{theorem}

\subsection{Lax-Phillips theory}
We will recall some results of the Lax-Phillips theory on the wave equation. Let's consider the following free wave equation%
\begin{equation}\label{El}
(E_{L})\left\{
\begin{array}{l}
\square u=0\text{\ on\ }\ \mathbb{R}\times \mathbb{R}^{3} \\
u(0,x)=u^{0}(x)\in H^{1}(\mathbb{R}^{3})\ \text{and }\ \partial
_{t}u(0,x)=u^{1}(x)\in L^{2}(\mathbb{R}^{3} ).%
\end{array}%
\right.
\end{equation}
It is well known that $(E_{L})$ admit a unique global solution $u$.\newline
If we denote $U(t)\varphi= \left(
             \begin{array}{c}
               u(t) \\
               \partial_{t}u(t) \\
             \end{array}
           \right)$, then U(t) forme strongly continuous and unitary group on H, generates by the unbounded operator $A=\left(
                                                                                                                          \begin{array}{cc}
                                                                                                                            0 & I \\
                                                                                                                            \Delta & 0 \\
                                                                                                                          \end{array}
                                                                                                                        \right)$ of the domain $D(A)=\{\varphi\in H, A\varphi \in H\}$
.\medskip\newline
Following Lax and Phillips, let's note the spaces of outgoing data
\begin{equation*}
D_{+,0}=\left\{ \varphi =(\varphi _{1},\varphi _{2})\in H\text{; }%
U(t)\varphi =0\text{ on }\left\vert x\right\vert \leq t,\text{ }t\geq
0\right\},
\end{equation*}%
and the space of incoming data
\begin{equation*}
D_{-,0}=\left\{ \varphi =(\varphi _{1},\varphi _{2})\in H\text{; }%
U(t)\varphi =0\text{ on }\left\vert x\right\vert \leq -t,\text{ }t\leq
0\right\},
\end{equation*}%
and for $R>0$, we set
\begin{equation*}
U(t)D_{+,0}=D_{+}^{R}=\left\{ \varphi =(\varphi _{1},\varphi _{2})\in H\text{; }%
U(t)\varphi =0\text{ on }\left\vert x\right\vert \leq t+R,\text{ }t\geq
0\right\},
\end{equation*}%
\begin{equation*}
U(t)D_{+,0}=D_{-}^{R}=\left\{ \varphi =(\varphi _{1},\varphi _{2})\in H\text{; }%
U(t)\varphi =0\text{ on }\left\vert x\right\vert \leq -t+R,\text{ }t\leq
0\right\}.
\end{equation*}%
In the following we write $D_{+}$ and $D_{-}$ instead of $D_{+}^{R}$ and $D_{-}^{R}$.\medskip\newline
These spaces satisfy the following properties

\begin{enumerate}
\item[a)] $D_{+}$ and $D_{-}\ $are closed in $H.$

\item[b)] $D_{+}$ and $D_{-}$ are orthogonal and%
\begin{equation*}
D_{+}\oplus D_{-}\oplus \left( \left( D_{+}\right) ^{\perp }\cap
\left( D_{-}\right) ^{\perp }\right) =H.
\end{equation*}
\item[c)] $\overline{\bigcup_{t\in\mathbb{R}}U(t)D_{\pm}}=H$ and $\bigcap_{t\in\mathbb{R}}U(t)D_{\pm}=\{0\}$.
\end{enumerate}
To measure local energy, Lax and Phillips introduce the operator $Z(t)=P^{+}U(t)P^{-}$ where $P^{+}$ (resp.$P^{-}$) is the orthogonal projection of H onto the orthogonal complement of $D_{+}$ (resp.$D_{-}$).

\subsection{Semi-group of Lax-Phillips adapted to Localized linear Klein-Gordon equation}
In this part we will show that the solution of (\ref{sys2}) is generated by a semi-group of contractions that we note ${U_{KG}(t):t\geq 0}$. We will then introduce the Lax-Phillips semi-group ${Z_{KG}(t)}$ adapted in our case.
\begin{proposition}
The operator
\begin{equation*}
A_{KG}=\left(
         \begin{array}{cc}
           0 & I \\
           \triangle - I & 0 \\
         \end{array}
       \right)
       \end{equation*}
       of domain
       \begin{equation*}
       D(A_{KG})=\{\varphi=(\varphi_{1},\varphi_{2})\in H\; \text {such that}\; A_{KG}\varphi \in H\}
       \end{equation*}
       is maximal dissipative.
       \end{proposition}

For $t \geq 0$ we set $Z_{KG}(t) = P^{+}U_{KG} (t)P^{-}$ where $P^{+}$ and $P^{-}$ are respectively
orthogonal projection on $(D_{+})^{\perp}$ and $(D_{-})^{\perp}$.\medskip \newline
The following proposition gives some properties of the operator $Z_{KG}(t)$.
\begin{proposition}
\begin{enumerate}
\item [1)]\bigskip $Z_{KG}(t)D_{+}=Z_{KG}(t)D_{-}=\left\{ 0\right\} $, for every $%
t $ $\geq 0.$

\item [2)] $Z_{KG}(t)$\ operates on $K=\left( D_{+}\right) ^{\perp }\cap \left(
D_{-}\right) ^{\perp }.$

\item [3)]$(Z_{KG}(t))_{t\geq 0}$ is a continuous semi-group on $K$.
\end{enumerate}
\end{proposition}
The arguments (with slight modifications) in the proof below are contained in \cite{1}. We include them for the convenience of the reader and to make the paper self-contained.
\begin{proof}
\begin{enumerate}
\item [1)]\bigskip Let $\varphi\in D_{-}$ then by definition of $P^{-}$ we have : $Z_{KG}(t)=0$.
Let $\varphi\in D_{+}$, since $D_{+}$ and $D_{-}$ are orthogonal then $P^{-}\varphi=\varphi$
and so to deduce that $Z_{KG}(t)\varphi = 0$, it is enough to verify $U_{KG}(t)D_{+}\subset D_{+}$. Let $\varphi \in D_{+}$ and $U_{KG}(t)\varphi=\left(
                                                                                                                                                 \begin{array}{c}
                                                                                                                                                   u(t) \\
                                                                                                                                                   \partial_{t}u(t) \\
                                                                                                                                                 \end{array}
                                                                                                                                               \right)$
the corresponding KG group where $u(t)$ is the solution of (\ref{sys2}). Since $\varphi\in D_{+}$ then $u(t,x)=0$ for $|x|\leq t+R$ and $t\geq 0$. As $supp(\chi_{1}u)\subset B_{R}$ then $u$ verifies :
\begin{equation*}
\left\{
\begin{array}{l}
\partial_{t}^{2}u-\triangle u=0\ \text{on\ }\ \mathbb{R_{+}}\times \mathbb{R}^{3},
\\
u(0)=\varphi_{1}\ \  \partial_{t}u(0)=\varphi_{2}\ \text{on\ }\mathbb{R}^{3}.
\end{array}%
\right.
\end{equation*}
According to the uniqueness of the solution of the equation (\ref{sys2}), we conclude that $U_{KG}(t)\varphi=U(t)\varphi$. Which gives $U_{KG}(t)\varphi \in D_{+}$ because $U(t)\varphi \in D_{+}$.\medskip
 \item [2)] Let $\varphi\in K=(D_{+})^{\perp}\cap (D_{-})^{\perp}$, show that $Z_{KG}(t)\varphi\in K$. It's easy to see that $Z_{KG}(t)\varphi \in (D_{+})^{\perp}$. In fact, let $g\in D_{+}$, we have
     \begin{equation*}
     (Z_{KG}(t)\varphi,g)=(P^{+}U_{KG}(t)\varphi,g)=(U_{KG}(t)\varphi,P^{+}g)=0,
     \end{equation*}
     which shows that $Z_{KG}(t)\varphi \in (D_{+})^{\perp}$. It remains to verify that $Z_{KG}(t)\varphi \in (D_{-})^{\perp}$. \medskip\\
     Let $g\in D_{-}$, we have
\begin{eqnarray*}
   (Z_{KG}(t)\varphi,g)&=& (P^{+}U_{KG}(t)\varphi,g)=(U_{KG}(t)\varphi,P^{+}g)=(U_{KG}(t)\varphi,g) \\
                       &=& (\varphi,U_{KG}^{*}(t)g).
\end{eqnarray*}
     To complete the proof of 2), we give the following lemma:
     \begin{lemma}
     Let $U_{KG}^{*}(t)$ the adjoint operator of $U_{KG}(t)$. Then $\forall \varphi \in D_{-}$ and $U_{KG}^{*}(t)\varphi=U(-t)\varphi.$
     \end{lemma}
     \begin{proof}
     Since $U_{KG}(t)$ is a semi-group generated by $A_{KG}$, then $U_{KG^{*}}(t)$ is a semi-group generated by $A_{KG}^{*}$.
     Let $g=(g_{1},g_{2})\in D_{-}$, we put
     \begin{equation*}
     U_{KG}^{*}(t)g =\left(
                       \begin{array}{c}
                         v_{1}(t) \\
                         v_{2}(t)\\
                       \end{array}
                     \right)
                     \end{equation*} such that
\begin{equation*}
 \left\{
                                                \begin{array}{ll}
                                                  \partial_{t}v_{1}=-v_{2}, & \hbox{} \\
                                                  \partial_{t}v_{2}=-\Delta v_{1}-\chi_{1}v_{1}, & \hbox{}
                                                \end{array}
                                              \right.
\end{equation*}
which implies
\begin{equation*}
\left\{
                \begin{array}{ll}
                  \partial_{t}^{2}v_{1}-\Delta v_{1}-\chi_{1}v_{1}=0, & \hbox{} \\
                  \partial_{t}v_{1}=-v_{2}. & \hbox{}
                \end{array}
              \right.
\end{equation*}
So, we have $U_{KG}^{*}(t)g=\left(
                              \begin{array}{c}
                                v_{1}(t) \\
                                -\partial_{t}v_{1}(t) \\
                              \end{array}
                            \right)=\left(
                                      \begin{array}{c}
                                        v(t) \\
                                        \partial_{t}w(t) \\
                                      \end{array}
                                    \right) $ with $v(t)$ solution of :
\begin{equation*}
\left\{
  \begin{array}{ll}
    \partial_{t}^{2}v-\Delta v-\chi_{1}v=0, &\hbox{on\;}\mathbb{R}_{+}\times\mathbb{R}^{3} ,\\
    (v(0),\partial_{t}v(0))=(g_{1},-g_{2}).
  \end{array}
\right.
\end{equation*}
Of the same $U(-t)g=\left(
                      \begin{array}{c}
                        w(t) \\
                        -\partial_{t}w(t) \\
                      \end{array}
                    \right)$ with $w(t)$ solution of
\begin{equation*}
\left\{
  \begin{array}{ll}
    \partial_{t}^{2}w-\Delta w=0,& \hbox{on\;}\mathbb{R}_{+}\times\mathbb{R}^{3}, \\
    (w(0),\partial_{t}w(0))=(g_{1},-g_{2}).
  \end{array}
\right.
\end{equation*}
Setting $\widetilde{v}(t)=v(-t)$ and $\widetilde{w}(t)=w(-t)$ with $t\leq 0$,
then $\widetilde{v}(t)$ and $\widetilde{w}(t)$ verify the following equations:
\begin{equation*}
\left\{
  \begin{array}{ll}
    \partial_{t}^{2}\widetilde{v}-\Delta \widetilde{v}-\chi_{1}\widetilde{v}=0, & \hbox{on\;}\mathbb{R}_{-}\times\mathbb{R}^{3}, \\
    (\widetilde{v}(0),\partial_{t}\widetilde{v}(0))=(g_{1},-g_{2}),
  \end{array}
\right.
\end{equation*}
\begin{equation*}
\left\{
  \begin{array}{ll}
    \partial_{t}^{2}\widetilde{w}-\Delta \widetilde{w}=0, & \hbox{on\;}\mathbb{R}_{-}\times\mathbb{R}^{3}, \\
    (\widetilde{w}(0),\partial_{t}\widetilde{w}(0))=(g_{1},g_{2}).
  \end{array}
\right.
\end{equation*}
Since $g\in D_{-}$ then $U(-t)g=0$ on $|x|\leq t+R$ and $t\geq 0$, then $\widetilde{w}(t)=0$ on $|x|\leq -t+R$ and $t\leq0$. But we have $supp\chi_{1}\subset B_{R}$ then by uniqueness of the solution we deduce that $\widetilde{w}(t)=\widetilde{v}(t)$ for $t\leq 0$, then $w(t)=v(t)$ for $t\geq 0$; and consequently $U_{KG}^{*}(t)g=U(-t)g$.
\end{proof}
Now let's go back to the second point proof of the proposition, according to Lax and Phillips \cite{18} $U(-t)D_{-}\subset D_{-}$ for all $t\geq 0$. We deduce then that $U_{KG}^{*}(t)g \in D_{-}$ and since $\varphi \in (D_{-})^{\perp}$ then $(Z_{KG}(t)\varphi,g)=(\varphi,U_{KG}^{*}(t)g)=0$ which shows that $Z_{Kg}(t)\varphi \in (D_{-})^{\perp}$.\medskip
\item [3)]Let $s$ and $t$ and $\varphi \in K$. We have
\begin{eqnarray*}
   Z_{KG}(t)Z_{KG}(s)\varphi &=& P^{+}U_{KG}(t)P^{-}Z_{KG}(s)\varphi\\
   &=& P^{+}U_{KG}(t)P^{+}U_{KG}(s)P^{-}\varphi   \\
   &=&P^{+}U_{KG}(t)P^{+}U_{KG}(s)\varphi.
\end{eqnarray*}
Since $P^{+}U_{KG}(t)P^{+}=P^{+}U_{KG}(t)$ (because $(P^{+}-I)$ is the orthogonal projection on $D_{+}$) then
$$
Z_{KG}(t)Z_{KG}(s)\varphi=P^{+}U_{KG}(t)U_{KG}(s)\varphi=Z_{KG}(t+s).
$$
\end{enumerate}
\end{proof}
The following lemma will be used to establish the exponential decay of the semigroup $Z_{KG}(t)$:
\begin{lemma}
We have\begin{itemize}
         \item [a)]$U_{KG}(t)(D_{-})^{\perp}\subset (D_{-})^{\perp}$ and $U(t)(D_{-})^{\perp}\subset (D_{-})^{\perp}$ for all $t\geq 0$.
         \item [b)]$U(t)(D_{-})^{\perp}\subset D_{+}$ for all $t\geq 2R$.
         \item [c)]If we put $M=U_{KG}(2R)-U(2R)$ then we have
\begin{equation*}
M\varphi=0 \quad \text{for}\; |x|\geq 3R \quad \text{and}\; \|M\varphi\|\leq 2 \|\varphi\|_{5R}.
\end{equation*}
         \item [d)]  $Z_{KG}(t)=P^{+}MU_{KG}(t-4R)MP^{-}$, $\forall t\geq 4R$.
       \end{itemize}
\end{lemma}
\begin{proof}
\begin{itemize}
  \item [a)] Let $\varphi \in (D_{-})^{\perp}$ and $g \in D_{-}$, we have $(U_{KG}(t)\varphi,g)=(\varphi,U_{KG}^{*}(t)g)$. As $U_{KG}^{*}(t)g\in D_{-}$, then $(U_{KG}(t)\varphi,g)=0$ and consequently $U_{KG}(t)\varphi \in (D_{-})^{\perp}$.\newline
In the same way we show the other inclusions.
  \item [b)]It suffices to show that $U(2R)(D_{-})^{\perp}\subset D_{+}$, according to the theory of representation (\cite{18}), the spaces $D_{-}$ and $D_{+}$ corresponding respectively to sub-spaces $L^{2}(]-\infty,-R]\times \mathcal{S}^{2})$ and $L^{2}([R,+\infty]\times \mathcal{S}^{2})$. Since the group $U(t)$ operate like translation on the right on $L^{2}$ then $U(2R)(D_{-})^{\perp}$ is represented by $L^{2}([R,+\infty]\times \mathcal{S}^{2})$ which proves the second point.
  \item [c)]Let $\varphi \in H$, by a domain of dependence argument (see \cite{18}), we see that
\begin{equation*}
U(t)\varphi=U_{KG}(t)\varphi \quad \text{on}\; |x|>t+R,\; t\geq 0.
\end{equation*}
In particular for $t=2R$,\; $U(2R)\varphi=U_{KG}(2R)\varphi$ on $|x|>3R$.\newline
Another application of the principle of domain of dependence shows that
\begin{equation*}
\|U(2R)\varphi\|_{3R}\leq \|\varphi\|_{5R}
\end{equation*}
and
\begin{equation*}
\|U_{KG}(2R)\varphi\|_{3R}\leq \|\varphi\|_{5R}
\end{equation*}
then
\begin{equation*}
\|M\varphi\|=\|M\varphi\|_{3R}\leq 2\|\varphi\|_{5R}.
\end{equation*}
  \item [d)] We have
\begin{eqnarray*}
  &&P^{+}MU(t-4R)MP^{-}\varphi\\
  &=& Z_{KG}(t)\varphi+P^{+}U(2R)U_{KG}(t-4R)U(2R)P^{-}\varphi \\
  &-&P^{+}U_{KG}(t-2R)U(2R)P^{-}\varphi-P^{+}U(2R)U_{KG}(t-2R)P^{-}\varphi.
\end{eqnarray*}
Using b), $U(2R)P^{-}\varphi \in D_{+}$, therefore the second and the third terms are equal to 0. Similiarly, using a) we deduce $U_{KG}(t-2R)P^{-}\varphi \in (D_{-})^{\perp}$.\\ Using again the argument in b) we get 
 $U(2R)U_{KG}(t-2R)P^{-}\varphi \in D_{+}$ which shows that the last term is equal to $0$.
\end{itemize}
\end{proof}
\subsection{Proof of Theorem \ref{thm3}}
In order to prove Theorem \ref{thm3} we will need the following theorem due to Nunes and Bastos which establish the polynomial decay of local energy. They proved this result for the linear Klein-Gordon equation in $\mathbb{R}^{n}$, $n\geq 1$, but we can see that, with slight modifications, the theorem and its proof remains valid in the context of our problem.
Let us state the theorem which is adapted in our case.
\begin{theorem} \cite[Nunes--Bastos]{36} \\
 Let $\Omega \subset \mathbb{R}^{n}$, $n\geq 1$ be a bounded domain. There exists positive constants $T_{0}>d(\Omega)$ and $K>0$ depending on $\Omega$, $c$ and $T_{0}$ such that for every $u^{0},u^{1}\in C_{0}^{\infty}(\mathbb{R}^{n})$ with $supp\; u_{0},$ $supp \;u_{1}\subset \Omega $, the solution $u$ to the Cauchy problem (\ref{1.1}) satisfies
\begin{equation}\label{eq3}
    E(u(t))\leq \frac{K}{t^{n}}\{\|u^{0}\|_{H^{1}(\Omega)}^{2}+\|u^{1}\|_{L^2(\Omega)}^{2}\}
\end{equation}
for every $t>T_{0}$.
\end{theorem}
\begin{remarks} \mbox{} \\
\begin{enumerate}
  \item The truncation function $\chi_{1}$ doesn't have impact on the proof of theorem; following \cite{33} and using the well known representation for the solutions of the wave equation see \cite{34}, and also an integral representation of Bessel's functions (see \cite[p. 437]{35}) we obtain the desired result.
  \item Using spectral approach, in our case, using the same method than Malloug \cite{38} we can find also the decay of the local energy for Klein-Gordon equation, which is of polynomial type.

  \item This result combined with the properties of semi-group ${Z_{KG}(t)}$ will allow us to show that the decay of energy is in fact, exponential.
\end{enumerate}
\end{remarks}
We comme back now to the proof of Theorem \ref{thm3}.\newline

For $\rho>0$, one poses $H_{\rho}=\{\varphi \in H;\varphi\; \text{with support in}\; B_{\rho}\}$. Let $\varphi \in H_{\rho}$ and $g \in D_{+}\cup D_{-}$, then
$g=0$ on $B_{R}$, so $(\varphi,g)=0$ and consequently $\varphi \in K$. On the other hand one knows that for very given $h$ of $H$, $P^{+}h=h$ on $B_{R}$, one thus obtains
\begin{equation*}
   U_{KG}(t)\varphi=Z_{KG}(t)\varphi \quad \text{on}\; B_{R}.
\end{equation*}
We have $\|U_{KG}(t)\varphi\|_{R}=\|Z_{KG}(t)\varphi\|_{R}\leq \|Z_{KG}(t)\varphi\|$. Thus to have the exponential decay of the local energy, it is enough to prove the exponential decay of $\|Z_{KG}(t)\|$.\newline
By applying the estimate (3.9) of Theorem 4, one chooses $\rho=5R$ and $T$ sufficiently large  such that
\begin{equation}\label{eq4}
    \|U_{KG}(t)g\|_{5R}\leq \frac{1}{8}\|g\|,\qquad g\in H_{5R}.
\end{equation}
Let $\varphi \in H$, by the previous lemma we have
\begin{eqnarray*}
 \|Z_{KG}(T+4R)\varphi\|&=& \|P^{+}MU_{KG}(T)MP^{-}\varphi\| \\
                       &\leq & \|MU_{KG}(T)MP^{-}\varphi\| \\
                        &=& \|MU_{KG}(T)MP^{-}\varphi\|_{3R} \\
                       &\leq & 2\|U_{KG}(T)MP^{-}\varphi\|_{5R} \\
                       &\leq & \frac{1}{4} \|MP^{-}\varphi\|\leq   \frac{1}{2}\|P^{-}\varphi\| \leq \frac{1}{2}\|\varphi\|                                                   . \end{eqnarray*}
One poses $T'=T+4R$; let $t>0$, there exist $k\in\mathbb{N}$ such that $ kT'\leq t\leq (k+1)T'$ and one deduces
\begin{equation*}
    \|Z_{KG}(t)\varphi\|\leq \|Z_{KG}(kT')\varphi\|\leq \|Z_{KG}(T')\varphi\|^{k}\leq \frac{1}{2^{k}}\|\varphi\|\leq C e^{-\alpha t}\|\varphi\|.
\end{equation*}
Thus, Theorem \ref{thm3} holds.
We are now ready to prove our main result.
\section{ Proof of Theorem \ref{exp} }\
Thanks to the Duhamel's Formula, the nonlinear equation (\ref{1.1}) for $u=U(t)\psi$ can be written as
\[
 u(t)=U_{L}(t)\psi+\int_{0}^{t}{U_{L}(t-s)I\chi_{2} u^{5}(s)ds.}
 \]
  Where $U_{L}(t)$ is the linear evolution group, $\chi_{2}=\chi_{2} (x)$ is the localizer, and $I$ is the mapping defined by $Iu=(0,u)$.
 Fix a ball $B_{R}$, an energy bound $E(\psi)\leq R_{0}$, and a smooth cut-off function $K$ that satisfies $K(x)=1$ on $supp(\psi)$. From now on  denotes $C$ any constants which may depend on $B_{R}$, $R_{0}$ and $K$.\medskip\\
 By the support property of $K$, we have
 $$
  Ku(t)=KU_{L}(t)\psi+\int_{0}^{t}{KU_{L}(t-s)KI\chi_{2} u^{5}(s)ds,}
$$
and using the fact that the local energy of $U_{L}(t)$ decay exponentially
$$\|KU_{L}(t)K\phi\|_{E}\leq Ce^{-\beta t} \|\phi\|_{E},$$
for some constant $\beta>0$, where $E$ denotes the energy space. Applying this estimate to the above integral identity, we obtain
\begin{equation}
    \|Ku\|_{L_{(t)}^{\infty}(E)}\leq \|KU_{L}(t)\psi \|_{L_{(t)}^{\infty}(E)}+C\int_{0}^{t}{e^{-\beta(t-s)}\| \chi_{2} u^{5}\|_{L_{(s)}^{1}(L^{2})}}ds.
    \label{eq2}
\end{equation}
For any $t\geq0$, where $L_{(s)}^{p}(X):=L^{p}(s,s+1;X)$, and the spatial domain $\mathbb{R}^{3}$ is \\ omitted.
By virtue of equation (\ref{eq1}) of Proposition 2.2 we deduces for any $\varepsilon>0$ that there exists $T>0$ such that  $$\|\chi_{2}^{1/{6}}u\|_{\mathrm{L^{\infty}(T,\infty,\mathrm{L}^{6})}}<\varepsilon.$$ If $\varepsilon $ is sufficiently small, we can bound any other space-time norms of the Strichartz type, just by (\ref{eq1}).\medskip\\ For example, we have $\| u\|_{L(T,\infty,L^{30})}^{5/{2}}<C$. By the H\"{o}lder's inequality, we have for any interval $I\subset(T,\infty)$
\begin{equation}
\|\chi u^{5}\|_{L^{1}(I,L^{2})} \leq C\|\chi^{1/{6}}u\|_{L^{\infty}(I,L^{6})}^{5/{2}}\|u\|_{L^{5/{2}}(I,L^{30})}^{5/{2}}\leq C\|\chi^{1/{6}}u\|_{L^{\infty}(I,L^{6})}^{5/{2}}.
 \end{equation}
 Since $|\chi^{1/{6}}u|\leq C |Ku|$ by the support property, the Sobolev inequality implies that
  \begin{equation}
\|\chi u^{5}\|_{L^{1}(I,L^{2})}\leq C \varepsilon^{3/{2}}\|Ku\|_{L^{\infty}(I,E)}.
\end{equation}
We apply these bounds to (\ref{eq2}), translating $t$ by $T.$ Denoting
\begin{equation}
f(t):=\|Ku\|_{L_{(t+T)}^{\infty}(E)}, \qquad      g(t):=\|KU_{L}(t-T)\|_{L_{(t+T)}^{\infty}(E)}.
\end{equation}
We obtain the following integral inequality
\begin{equation}
f(t)\leq g(t)+C\varepsilon^{3/{2}}\int_{0}^{t}e^{-\beta(t-s)}f(s)ds
\end{equation}
then
\begin{equation*}
  f(t)\leq Ce^{-\beta t}+C\varepsilon^{{3}/{2}}\int_{0}^{t}e^{-\beta(t-s)}f(s)ds,
\end{equation*}
which is equivalent to
\begin{equation*}
  e^{\beta t}f(t) \leq C+C\varepsilon^{{3}/{2}}\int_{0}^{t}e^{\beta s}f(s)ds.
\end{equation*}
By virtue of Gronwall lemma, we obtain
\begin{equation*}
   f(t) \leq Ce^{(c\varepsilon^{{3}/{2}}-\beta)t},\quad \text{for}\; t \geq 0.
\end{equation*}
We choose $\varepsilon$ such that $c\varepsilon^{{3}/{2}}-\beta<-\frac{\beta}{2}$, we obtain the exponential decay for $f(t)$, which implies that of $E(Ku(t))$.

\end{document}